\documentclass[letterpaper, 10 pt, conference]{ieeeconf}  

\IEEEoverridecommandlockouts                              
\usepackage{amssymb}
\usepackage{amsmath}

\usepackage{amsthm}
\usepackage{algorithm,algorithmic}
\usepackage{graphicx}
\usepackage{bbm}
\usepackage{cite}

\newtheorem{theorem}{Theorem}

\usepackage{comment}

\usepackage{multirow}

\title{\LARGE \bf
Finite-Horizon Markov Decision Processes\\ with State Constraints
}

\author{{Mahmoud El Chamie $\ $}
and
{$\, $ Beh\c{c}et A\c{c}{\i}kme\c{s}e 
\thanks{$^{*}$ The authors are with the University of Texas at Austin, department of Aerospace Engineering and Engineering Mechanics, 210 E. 24th St., Austin, TX 78712 USA. Emails: {\em melchami@utexas.edu} and {\em behcet@austin.utexas.edu}}}
}

\begin{document}

\maketitle
\thispagestyle{empty}
\pagestyle{empty}

\begin{abstract}
Markov Decision Processes (MDPs) have been used to formulate many decision-making problems   in science and engineering. The objective is to synthesize  the best decision (action selection) policies to maximize expected rewards (minimize costs) in a given  stochastic dynamical environment. In many practical scenarios (multi-agent systems, telecommunication, queuing, etc.), the decision-making problem can have  state constraints that must be satisfied, which leads to Constrained MDP (CMDP) problems. In the presence of such state  constraints, the optimal policies can be very hard to characterize.  This paper introduces  a new approach for finding non-stationary randomized policies for finite-horizon CMDPs. An efficient algorithm based on Linear Programming (LP) and duality theory is proposed, which  gives the convex set of all feasible policies and   ensures that  the expected total reward is above a computable lower-bound. The resulting decision policy is a randomized policy, which is  the projection of the unconstrained deterministic MDP policy on this convex set. To the best of  our knowledge, this is the first  result in state  constrained MDPs to give an efficient algorithm for generating finite horizon randomized policies for CMDP with optimality guarantees. A simulation example of a swarm of autonomous agents running MDPs is also presented to demonstrate the proposed CMDP solution algorithm. 
\end{abstract}

\section{Introduction}

Markov Decision Processes (MDPs) have been used to formulate many decision-making problems  in a variety of  areas of science and engineering \cite{Parkes:2003Anm,Dolgov:2006Res,Doshi:2004Dyn}. MDPs can also be  useful    in modeling decision-making problems for stochastic  dynamical systems where the dynamics cannot be fully captured by using first principle formulations. MDP models can be constructed by utilizing the available measured  data, which allows construction of the state transition probabilities. Hence MDPs play a critical role in big-data analytics. Indeed  very popular   methods of machine learning  such as reinforcement learning and its variants  \cite{sutton1998introduction}\cite{szepesvari2010} are built on the MDP framework.    With the increased interest and efforts in Cyber-Physical Systems (CPS), there is  even more interest in MDPs to facilitate rigorous construction of innovative hierarchical decision-making architectures, where MDP framework can integrate physics-based models with data-driven models.  Such decision architectures can utilize a systematic approach to bring physical devices together with software  to benefit many emerging engineering  applications, such as autonomous systems.    

In many applications \cite{feinberg2002handbook}\cite{altman:inria-00072663}, MDP models  are used to compute optimal  decisions when future actions contribute to the overall mission performance. 
Here  we consider MDP-based  sequential stochastic decision-making models  \cite{Puterman:1994Mar}. 
An MDP model is composed of a set of  time epochs, actions, states, and immediate rewards/costs. Actions transfer the system in a stochastic manner from one state to another and rewards are collected based on the actions taken at the corresponding states. Hence MDP models provide  analytical descriptions  of stochastic processes with  state and action spaces, the state transition probabilities as a function of actions, and with rewards as a function of the states and actions. The objective is to synthesize the best decision (action selection) policies to maximize expected rewards (minimize costs) for a given MDP. 
  It is well-known that optimal policies  must be stationary deterministic  when then the environment is stationary \cite{Puterman:1994Mar} and when there are no state constraints. We  present  new  results that aim to increase fidelity of MDPs for  decision-making by incorporating a  general class of state constraints in the MDP models, which then  lead to randomized action selection policies. 

In this paper, we study the problem of finding non-stationary randomized policy solutions for finite-horizon constrained MDPs (CMDPs). We consider a finite state MDP with randomized action sets. We give an efficient algorithm based on Linear Programming (LP) and duality theory of convex optimization \cite{Boyd:2004:CO:993483} that optimizes over  the convex set of all feasible policies and  guarantees the expected total reward to be above a  computable lower bound. Then the proposed policy is the projection of the unconstrained MDP policy on this  convex set. To best of our knowledge, this is the first result in state constrained MDP problems that gives an efficient algorithm for generating finite horizon randomized policies for CMDP with reward/cost guarantees. Another advantage of the proposed solution is that it is independent of initial state of the system. Thus it can be solved offline and implemented in large-scale systems of multi-agent systems.
 
 \section{Related Work}
 In MDPs, state constraints can be utilized   in several ways.  They can be used to  handle multiple design objectives where decisions are computed to maximize rewards for one  of the objectives while guaranteeing the value of the other objective to be   within a desired value \cite{Dolgov:2005Sta}. The constraints can also be  imposed by the environment  (e.g., safety constraints imposed by a mission  as in multi-agents autonomous systems \cite{behcet_tac15}), or telecommunication applications  \cite{altman1999constrained}. In these state constrained MDPs,   the calculation of optimal policies can be much more difficult, so the constraints are usually relaxed with the hope that the resulting decisions would still provide feasible solutions. However, in some applications, these constraints are critical \cite{ari_03,ari_acc06,nazli_acc14,behcet_tac15}. Consider an example of exploring  a disaster  area for search and rescue by using multiple autonomous vehicles.
  Suppose agents are running MDP policies to explore  the area based on a priori knowledge of the potential survivors. Due to safety conditions, vehicles may not be allowed to visit certain regions often, which can impose strict constraints on the  probability of having  a vehicle in such regions at any given time epoch. These safety considerations  can be formulated  as constraints on the probability distribution of the agent state, e.g., an inequality constraint for MDPs with discrete state and action spaces of finite cardinality
 $$
 B \mathbf{x}_t \leq \mathbf{d} \quad \ \forall \, t\geq 0,
 $$
 where $B \in \mathbb{R}^{n \times n}$ is a  matrix and $\mathbf{d} \in \mathbb{R}^ n$ is a vector that describes the safety constraints, $\mathbf{x}_t \in \mathbb{R}^n$ is the vector whose elements $x_t[i], i=1,...,n,$ are   the discrete probabilities for an agent to be in state $i$ at time $t$, and $\leq$ denotes element-wise inequality. 
    This paper aims to incorporate such safety  constraints into MDP formulations, that is,  optimal decision policies are synthesized  within the constraints. The state constraint above also allows direct relationship with the chance-constrained decision making, eg., chance constrained motion planning \cite{Blackmore:2010Apr}, where state constraints must be satisfied  with a prescribed probability. 
 MDPs with constraints has recently been applied to path planning in robotics applications \cite{Feyzabadi:2014Ris}. Beyond being able to ensure  safety, these constraints can provide advancements   in machine learning methods \cite{Dana:2013, Geibel:2006Rei, Elchamie:2014New}. Incorporating the  knowledge of physical constraints can potentially improve the estimates of  MDPs by better utilizing the real-time data, i.e., by searching the MDP parameters in smaller and better constrained feasible  sets. 

Previous research has focused on finding infinite-horizon stationary policies for constrained MDPs.  Due to the constraints, the optimal policies might no longer be deterministic and stationary \cite{Feinberg:1995Con}. \cite{Haviv:1996Onc} gives an example of a transient multi-chain MDP with state constraints and shows that the Bellman principle fails to hold and the optimal policy is not necessarily stationary. In the presence of constraints, randomization in the actions can then be necessary for obtaining optimal policies \cite{Nain:1986Opt}\cite{Frid:1972Ono}.  For stationary policies to be optimal, specific  assumptions on the underlying Markov chain are often introduced \cite{Altman:1991Mdp}.
Stationary policies for these specific models can be found by using algorithms based on Linear Programming (LP) 
\cite{Hordijk:1984Con} or Lagrange multipliers (\cite{Altman:1995The}\cite{Beutler:1986Tim}). 
 However, finding optimal policies in the broader class of   randomized policies  for CMDPs can be very  expensive computationally and there is no previously   known algorithm for the general case\footnote{See \cite{Feinberg:1996Con} for  very specific   examples where solutions can be  obtained.} \cite{Dolgov:2005Sta} \cite{Lee:2014Ext}.

\section{Preliminaries and Notation}
\subsection{States and Actions}
 Let the set $S=\{ 1, \dots , n\}$ be the set of states (note that $S$ is finite of cardinality $|S|=n$). Let us define $A_s=\{ 1, \dots , p\}$  to be the set of actions available in state $s$ (without loss of generality the number of actions does not change with the state, i.e., $|\mathcal{A}_s|=p$ for any $s\in S$). We consider a discrete-time system where actions are taken at different decision epochs. Let $s_t$ and $a_t$ be respectively the state and action at the $t$-th decision epoch. 

\subsection{Decision Rule and Policy}
We define a decision rule $D_t$ at time $t$ to be the following randomized function $D_t : S \rightarrow A_S$ that defines for every state $s\in S$ a random variable $D_t(s)\in A_{s}$ with a probability distribution defined on $\mathcal{P}(A_s)$ as follows $q_{D_t(s)}(a)=\text{Prob}[D_t=a|s_t=s]$ for any action $a\in A_s$. 
Let $$\pi =(D_1, D_2, \dots, D_{N-1})$$ be the policy for the decision making process given that there are $N-1$ decision epochs. Note that this decision rule has a Markovian property because it depends only on the current state. Indeed this paper considers only the Markovian policies, and  history dependent policies \cite{Puterman:1994Mar} are not considered.  

\subsection{Rewards}
Given a state $s\in S$ and action $a\in A$, we define the reward $r_t(s_t,a_t)\in \mathbb{R}$ to be any  real number and let $\mathcal{R}$ be the set having these values. 
With a little abuse of notation, we define the expected  reward for a given  decision rule $D_t$ at time $t$ to be $$r_t(s)=\mathbb{E}[r_t(s,D_t(s))]=\textstyle \sum _{a\in A_{s}}q_{D_t(s)}(a)r_t(s,a),$$ and the vector $\mathbf{r}_t\in \mathbb{R}^n$ to be the vector with the expected rewards for each state. Since  there are $N-1$ decision epochs,  there are $N$ reward stages and the final stage reward is given by $r_N(s_N)$ (or $\mathbf{r}_N$ the vector whose entries are the final rewards for each state).

\subsection{State Transitions}
We now define the transition probabilities as follows, $p_t(j|s,k)= \text{Prob}[s_{t+1}=j|s_t=s, a_t=k]$, and let $\mathcal{G}$ be the set of these transition probabilities. Let 
$$p_t(j|s,d(s))=\textstyle \sum_{a\in A_s}q_{D_t(s)}(a)p_t(j|s,a),$$
then the elements of the transition matrix $M_t\in \mathbb{R}^{n,n}$ are:
$$M_t[i,j]=\text{Prob}[s_{t+1}=i|s_t=j]=p_t(i|j,d(j)).$$
Let $x_t[i]=\text{Prob}[s_t=i|s_1]$ to be the probability of being at state $i$ at time $t$, and $\mathbf{x}_t\in \mathbb{R}^m$ to be a vector having these elements. Then the system evolves according to the following recursive equation $\mathbf{x}_{t+1}=M_t\, \mathbf{x}_t, \quad t=1,\dots,N$.

\subsection{Markov Decision Processes (MDPs)}
Let $\gamma \in [0,1]$ be the discount factor, which represents the importance of a current reward in comparison to future possible rewards. We will consider $\gamma =1$ throughout  the paper, but the results  are not affected and remain applicable after a suitable scaling when $\gamma <1$. 

A discrete MDP is a 5-tuple $(S, A_S, \mathcal{G}, \mathcal{R}, \gamma )$ where $S$ is a finite set of states, $A_s$ is a finite set of actions available for  state $s$, $\mathcal{G}$ is the set that contains the transition probabilities   given the current state and current action, and $\mathcal{R}$ is the set of rewards at time $t$  due to the current state and  the action.

\subsection{Performance Metric}
 For a policy to be better than another policy we need to define a performance metric. We will use the expected discounted total reward for our performance study, 
 $$v_N^\pi=\mathbb{E}_{\mathbf{x}_1}\left[\sum_{t=1}^{N-1}r_t(X_t,D_t(X_t))+ r_N(X_N)\right],$$
where the expectation is conditioned on knowing the probability distribution of the initial states (i.e., knowing $\mathbf{x}_1\in \mathcal{P}(S)$ where $x_1[i]=\text{Prob}[s_1=i]$). For example if the agent was in state $s$ at $t=1$, then $\mathbf{x}_1=\mathbf{e}_s$ where $\mathbf{e}_s$ is a vector of all zeros except for the $s$-th element which is equal to 1.  It is worth noting that in the above expression, both $X_t$ and $D_t$ are random variables. 

\section{Optimal Markovian Policy Synthesis Problem}
The optimal policy $\pi ^*$ is given as the policy that maximizes the performance measure, $\pi ^* = \text{argmax}_{\pi} v_N^\pi$, and $v^*_N$ to be the optimal value, i.e., $v_N^* = \text{max}_{\pi} v_N^\pi$.
Note that this maximization is unconstrained and the optimization variables are $q_{D_t(s)}(a)$ for any $s\in S$ and $a\in A_s$.\footnote{Since $v^\pi_N$ is continuous in the decision variables that belong to a closed and bounded set, then the $\max$ is always attained and $\text{argmax}$ is well defined.} The \emph{backward induction} algorithm \cite[p.~92]{Puterman:1994Mar} based on dynamic programming gives the optimal policy as well as the optimal value by using Algorithm \ref{alg1}.

\begin{algorithm}
\caption{Backward Induction: Unconstrained MDP Optimal Policy}
\label{alg1}
\begin{algorithmic} [1]
 \STATE {\bf Definitions:} For any state $s\in S$, we define $V_t^\pi(s)=\mathbb{E}_{\mathbf{x}_t=\mathbf{e}_s}\left[\sum_{k=t}^{N-1}r_k(X_k,D_k(X_k))+ r_k(X_N)\right]$ and $V_t^*(s)=\text{max}_{\pi} V_t^\pi$ given that $s_t=s$.
\STATE Start with $V^*_N(s)=r_N(s)$
\STATE for $t=N-1, \dots, 1$ given $V_{t+1}^*$ for all $s\in S$ calculate the optimal value
  $$V_{t}^*(s)=\max _a \big\{ r_t(s,a) + \sum _{j\in S}p_t(j|s,a)V_{t+1}^*(j)\big\}$$
  and the optimal policy is defined by $q_{D_t(s)}(a)=1$ for $a=a_{t}^*(s)$ given by: 
  $$a_{t}^*(s)=\text{argmax} _a \big\{ r_t(s,a) + \sum _{j\in S}p_t(j|s,a)V_{t+1}^*(j)\big\}$$
\STATE {\bf Result:} $V_1^*(s_1)=v_N^*$ where $s_1$ is the initial state.
\end{algorithmic}
\end{algorithm}

Algorithm \ref{alg1} solves the optimal policy selection  in the absence of  constraints on the expected state vector $\mathbf{x}_t$ for $t=1,\dots, N$. Next we introduce state constraints as follows
$$B\mathbf{x}_t\leq \mathbf{d}, \text{ for }t=1, \dots, N,$$
where $\leq$ denotes the element-wise inequalities, and $\mathbf{d}$ is a vector giving upper bounds on the state/transition probabilities.
 These state constraints   lead to correlations between decision rules at different states and the backward induction algorithm cannot then be used  to find optimal policy when the state constraints exist. Even finding a feasible policy  can be very challenging. We refer to this problem as Constrained MDP (CMDP).
 
The optimal  policy synthesis  problem with  constraints on  $\mathbf{x}_t$ can then be written as follows,
\begin{equation}
\begin{array}{rrclcl}
\displaystyle \max_{Q_1,\dots, Q_{N-1}} & \multicolumn{3}{l}{v_N^\pi} \\
\textrm{s.t.} & B \mathbf{x}_t & \leq & \mathbf{d}, \text{ for } t=1,\dots, N-1\\
& Q_t\mathbf{1} & = & \mathbf{1},\text{ for } t=1, \dots, N-1 \\
& Q_t & \geq & 0,  \text{ for } t=1, \dots, N-1, \\
\end{array}
\end{equation}
where $Q_t \in \mathbb{R}^{n,p}$ is the matrix of  decision  variables $q_t(s,a)=q_{D_t(s)}(a)$.  The last two sets of constraints guarantee that the variables define probability distributions. $B$ is a constant matrix, which is  assumed to be the identity $B=I_n$ (but the following discussion easily extends to any matrix $B$). Without the first set of constraints, the rows in $Q_t$ are independent and they are not correlated. With the added first set of constraints (that are non-convex because $\mathbf{x}_t=M_{t-1}\dots M_2M_1\mathbf{x}_1$, where each of the matrices $M_i$ is a linear function of the variables $Q_i$) correlation would exist between the rows of $Q_i$'s and the backward induction that leverages the independence of the rows of $Q_i$ cannot be applied as usual. The next section introduces a dynamic programming based algorithm for the above problem.

\section{Dynamic Programming (DP)  Approach to  Markovian Policy Synthesis}
In this section, we  transform the MDP problem into a deterministic Dynamic Programming (DP) problem, give the equivalence with the  unconstrained case and discuss how to solve  the (more complicated) state constrained problem.
First note  that the performance metric can be written as follows:
\begin{align*}
v_N^\pi&=\mathbb{E}_{\mathbf{x}_1}\left[\sum_{t=1}^{N-1}r_t(X_t,D_t(X_t))+ r_t(X_N)\right]\\
&=\sum _{t=1}^{N-1}\mathbb{E}_{\mathbf{x}_1}[r_t(X_t,D_t(X_t))] + \mathbb{E}_{\mathbf{x}_1}[r_t(X_N)]\\
&=\sum _{t=1}^{N}\mathbb{E}_{\mathbf{x}_1}[\mathbf{e}_{X_t}^T\mathbf{r}_t]=\sum _{t=1}^{N}\mathbf{x}_t^T\mathbf{r}_t,
\end{align*}
where the last equality utilized the fact that $\mathbb{E}_{\mathbf{x}_1}[X_t]=\mathbf{x}_t$.

Next we present the DP  formulation. Let $\mathbf{x}_t$ to be an element of the extended state space $\mathcal{S}=\mathcal{P}(S)$ (where $\mathcal{P}(.)$ denotes the probability space). The discrete-time dynamical system describing the evolution of the density $\mathbf{x}_t$ can then be given by
\begin{align*}
\mathbf{x}_{t+1} &=f_t(\mathbf{x}_t,Q_t) \text{ for } t=1,\dots , N-1,
\end{align*} 
such that $f_t(\mathbf{x}_t,Q_t) = M_t(Q_t)\mathbf{x}_t$ where $M_t(Q_t)$ is a column stochastic  matrix linear in the optimization variables $Q_t$. It is important to note that  the elements of the $i$-th \emph{column} in $M_t$ are  linear functions of  only the elements in the $i$-th \emph{row} of $Q_t$,  not all elements of $Q_t$. The above dynamics show that the probability distribution  evolves deterministically. Our policy $\pi=(D_1, \dots, D_{N-1})$ consists of a sequence of functions that map states $\mathbf{x}_t$ into controls $Q_t=D_t(\mathbf{x}_t)$  such  that $D_t(\mathbf{x}_t) \in \mathcal{C}(\mathbf{x}_t)$ where $\mathcal{C}(\mathbf{x}_t)$ is the set of feasible controls. In case of absence of the constraints ($B \mathbf{x}_t  \leq  \mathbf{d}$), $\mathcal{C}(\mathbf{x}_t)$ is independent of $\mathbf{x}_t$ and all admissible controls belong to the same convex set $\mathcal{C}$ for all states.

The additive reward per stage is defined as $g_N(\mathbf{x}_N)=\mathbf{x}_N^T\mathbf{r}_N$ and 
$$g_t(\mathbf{x}_t, Q_t)=\mathbf{x}_t^T\mathbf{r}_t , \text{ for } t=1, \dots, N-1.$$ 
The DP  algorithm  calculates the optimal value $v^*_N$ (and policy $\pi ^*)$ as follows \cite[Proposition 1.3.1, p.~23]{Bertsekas:2005Dyn}:
\begin{algorithm}
\caption{Dynamic Programming (DP)}
\label{alg2}
\begin{algorithmic} [1]
 \STATE Start with $J_N(\mathbf{x})=g_N(\mathbf{x})$
\STATE for $t=N-1,\dots , 1$
$$J_t(\mathbf{x})=\max _{Q_t\in \mathcal{C}(\mathbf{x})}\left\{g_t(\mathbf{x}, Q_t)+J_{t+1}(f_t(\mathbf{x},Q_t))\right\}.$$
\STATE {\bf Result:} $J_1(\mathbf{x})=v^*_N$.
\end{algorithmic}
\end{algorithm}

\begin{proof}[\bf Remark]
There are several  difficulties in applying the DP Algorithm \ref{alg2}. Note that in the expression $J_{t+1}(f_t(\mathbf{x},Q_t))$ used in the algorithm, $Q_t$ is an optimization variable. 
For a given $Q_t$ and $\mathbf{x}$, numerical methods can be used to compute the value of  $J_{t+1}$. 
But since $Q_t$ itself is an optimization variable,  the solution of the optimization problem in line $2$ of Algorithm $2$ can be very hard. 
In some special cases, for example when $J_t(\mathbf{x})$  can be expressed analytically  in a closed from, the solution complexity can be reduced significantly, as in the unconstrained MDP problems. 
\end{proof}


\subsection{Solving Unconstrained MDPs  by  DP}
%
Here we use the DP algorithm to derive well-known results on optimal MDP policies \cite{Puterman:1994Mar}. Even though the DP approach is not new (i.e., Algorithm \ref{alg2} itself is derived from theory of operators), its application  to finite-state MDPs will provide new insights. Specifically, when finite-state MDPs are subject to state constraints,  the existing theory cannot be readily applied. In that case, we show that  the DP algorithm can  still provide useful results. Therefore, the purpose of this section is to apply the DP algorithm for  unconstrained MDPs to obtain well-known results, and set up its use  for  more complicated finite-state constrained MDP problems.

Next we present the closed-form  solution of the unconstrained MDPs via  Algorithm \ref{alg1}. In the absence of state constraints, the set of admissible actions  at time $t$, denoted  by $\mathcal{C}(\mathbf{x}_t)=\mathcal{C}$, can be described as 
 $$Q_t\mathbf{1}= \mathbf{1}, \text{ and } Q_t\geq 0.$$
 Note that each row of $Q_t$  represents the action choice   probabilities for a given state, i.e.,    
 $Q_t^{i}\in \mathcal{C}^i$ for $i=1, \dots, n$ where $Q_t^{i}$ is the $i$-th row in $Q_t$ and $\mathcal{C}^i$ is the set of  row vectors of probabilities having dimension $|A|$. We can now apply the DP Algorithm \ref{alg2} by letting $J_{N}(\mathbf{x})=\mathbf{x}^T\mathbf{r}_N$, and iterating backward from $t=N-1$ to $t=1$ as follows  
\begin{align*}
J_{t}(\mathbf{x})&= \max _{Q_t\in \mathcal{C}}\{\mathbf{x}^T\mathbf{r}_t +J_{t+1}(M_t\mathbf{x})\}\\
&=\max _{Q_t^{i}\in C^{i}, i=1,\dots, m} \{\mathbf{x}^T\mathbf{r}_t+\mathbf{x}^TM_t^TV^*_{t+1}\}\\
&=\max _{Q_t^{i}\in C^{i}, i=1,\dots, m} \{ \sum _i x[i] (r_t(i,Q_t^{i})+M^{Ti}_t(Q_t^{i})V^*_{t+1})\}\\
&=\sum _i x[i]\left(\max _{Q_t^{i}\in C^{i}}\left\{ r_t(i,Q_t^{i})+M^{Ti}_t(Q_t^{i})V^*_{t+1}\right\}\right),
\end{align*}
 where $V^*_{t+1}$ is the optimal value function computed by Algorithm \ref{alg1}, and $M^{Ti}_t(Q_t^{i})$ indicates the transpose of the $i$-th column of $M_t$ which is a linear function of the variables in the $i$-th row of $Q_t$. The last equality is due to the fact that $x[i]\geq 0$ for all $i$ and the value function is  separable in terms of the optimization variables. The maximization inside the parenthesis in the last equation gives  $V^*_{t}[i]$, and hence 
 $$J_{t}(\mathbf{x})=\sum _i x[i]V^*_{t}[i] = \mathbf{x}^T V^*_{t},$$
 which has a closed-form solution as function of $\mathbf{x}$.  This discussion justifies that the calculation of $V^*_{t}$ for $t=N,\dots ,1$ in Algorithm \ref{alg1} is sufficient for finding the optimal value of the MDP given by $v^*_N=J_1(\mathbf{x}_1)=\mathbf{x}_1^TV^*_{1}$.
 
 \begin{proof}[\bf Remark]
 $V^*_{1}$ obtained via Algorithm \ref{alg1} leads to a  \emph{deterministic} \emph{Markovian} policy, which also defines an optimal policy for the unconstrained MDP, i.e., the policy that maximizes  the total expected reward. It must be emphasized  that, if state constraints were present, then Algorithm $1$ does not necessarily  yield an optimal, or even a feasible,  policy.
 \end{proof}   
  
 \subsection{Solving State Constrained MDPs by DP}
When the state constraints are present,  $J_{t}(\mathbf{x})$ does not have a closed-form solution, and hence finding an optimal (even a feasible) solution is challenging. 
This section presents  a new algorithm, Algorithm \ref{alg3},  to compute a feasible solution of the state constrained MDPs with  lower bound  guarantees on the  expected reward.

\begin{algorithm}
\caption{Backward Induction: State constrained MDP}
\label{alg3}
\begin{algorithmic} [1]
 \STATE {\bf Definitions:} $Q_t$ for $t=1,..., N-1$ are the optimization  variables, describing the decision policy. Let $\mathcal{X}\!=\!\{\mathbf{x}\!\in \! \mathbb{R}^n\!: 0\! \leq \! \mathbf{x}\!\leq\! \mathbf{d}, \, \mathbf{1}^T \mathbf{x}\!=\! 1  \}$ and    $\mathcal{C}\!=\! \underset{\mathbf{x}\in \mathcal{X}}{\cap} \mathcal{C}(\mathbf{x})$ with 
\begin{eqnarray*}
 \mathcal{C}(\mathbf{x}) &\!=\! &\left\{Q \! \in \! \mathbb{R}^{n\times p}\!:  Q\mathbf{1}\!=\! \mathbf{1},\, Q \! \geq \! 0,     M(Q)\mathbf{x}\!\leq \!\mathbf{d}  \right\}\!
 \end{eqnarray*}
 where $M(Q)$ is the transition matrix linear in $Q$. 
 
\STATE Set $\hat{U}_N=\mathbf{r}_N$.
\STATE For $t\!=\! N-1, \dots, 1$, given $\hat{U}_{t+1}$, compute the policy
\begin{equation}\label{eq:Qhat}
\hat{Q}_{t}=\underset{Q\in \mathcal{C}}{\text{argmax}} \ \min _{\mathbf{x}\in \mathcal{X}}\hspace*{0.2cm} \mathbf{x}^T \left(\mathbf{r}_{t}(Q)+M_{t}(Q)^T\hat{U}_{t+1}\right),
\end{equation}
and the vector of expected rewards
  $$\hat{U}_{t}=\mathbf{r}_{t}(\hat{Q}_t)+M_{t}(\hat{Q}_t)^T\hat{U}_{t+1}.$$
\STATE {\bf Result:}  $v_N^*\geq \mathbf{x}_1^T\hat{U}_{1}$
\end{algorithmic}
\end{algorithm}
 
\begin{theorem}\label{thm:1}
 Algorithm \ref{alg3} provides a feasible  policy for the state constrained MDP that guarantees the expected total reward to be greater than a lower bound $R^\#=\mathbf{x}_1^T\hat{U}_{1}$, i.e.,
 $v_N^*\geq R^\#$.
\end{theorem}
\begin{proof}
The proof is based on applying the DP Algorithm \ref{alg2}. Letting $J_{N}(\mathbf{x})=\mathbf{x}^T\mathbf{r}_N$,  it will be shown by induction that $J_{t}(\mathbf{x})\geq \mathbf{x}^T\hat{U}_t$. It is true for $t=N$. Now supposing that it is true for $t+1$, let's prove it true for $t$. We have from Algorithm \ref{alg2} that
 \begin{align*}
J_t(\mathbf{x})&=\max _{Q_t\in \mathcal{C}(\mathbf{x})}\left\{g_t(\mathbf{x}, Q_t)+J_{t+1}(f_t(\mathbf{x},Q_t))\right\}\\
&=\max _{Q_t\in \mathcal{C}(\mathbf{x})}\left\{\mathbf{x}^T\mathbf{r}_t+J_{t+1}(M_t\mathbf{x})\right\}\\
&\geq \max _{Q_t\in \mathcal{C}(\mathbf{x})}\left\{\mathbf{x}^T\mathbf{r}_t+\mathbf{x}^TM_t^T\hat{U}_{t+1}\right\}\\
&\geq \max _{Q_t\in \mathcal{C}}\left\{\mathbf{x}^T(\mathbf{r}_t+M_t^T\hat{U}_{t+1})\right\}\\
&\geq \mathbf{x}^T\hat{U}_t,
\end{align*}
where in the third line we applied the induction hypothesis and the last line by the definition of $\hat{Q}_t$. Since $J_1(\mathbf{x}_1)=v_N^*$ (line 3 in Algorithm \ref{alg2}), this ends the proof. 
\end{proof} 
 
Therefore, by calculating $\hat{U}_t$ for $t=N-1, \dots, 1$ from Algorithm \ref{alg3}, we can find a
policy that gives minimum guarantees on the total expected reward, namely  $v_N^{\hat{\pi}}\geq R^{\#}$, where $\hat{\pi}=(\hat{Q}_1, \hat{Q}_2, \dots , \hat{Q}_{N-1})$.


\subsection{LP formulation of $\hat{Q}_t$}
The calculation of  $\hat{Q}_{t}$ for $t=1, \dots, N-1$  in equation \eqref{eq:Qhat} is the main challenge in the application of Algorithm \ref{alg3}. This section describes a linear programming approach  to  the computation of $\hat{Q}_{t}$ solution in every iteration loop of Step 3 in Algorithm \ref{alg3}.
From the algorithm,  $\hat{Q}_t$ is given as follows:
\begin{equation}\label{eq:Qt}
\begin{split}
\hat{Q}_{t}&=\underset{Q_{t}\in \mathcal{C}}{\text{argmax}} \ \min _{\mathbf{x}\in \mathcal{X}}\hspace*{0.2cm} \mathbf{x}^T U_{t}(Q_{t}),  \\
 \mbox{where } & U_{t}(Q_{t}) =\mathbf{r}_{t}(Q_t)+M_{t}(Q_t)^T\hat{U}_{t+1}  \\ 
\mbox{with }\hspace*{0.2cm}&\hat{U}_{N} =\mathbf{r}_N \ \  \mbox{and}  \ \ \hat{U}_{t}\!=\! U_{t}(\hat{Q}_{t}),  \ \  t\!=\! N\!-\!1, \dots, 1.  
\end{split}
\end{equation}
Let $R_t\in \mathbb{R}^{m,p}$ be the matrix having the elements $r_t(s,a)$ where $s\in S$ and $a\in A$. Let $G_{t,k}$ be the transition matrix having the elements $G_{t,k}[i,j]=p_t(i|j,k)$.
Then we have 
\begin{equation}\label{eq:MR}
M_t(Q_t)\!=\! \sum _kG_{t,k}\odot \left(\mathbf{1}(Q_t\mathbf{e}_k)^T\right)
\mbox{and} \ \  \mathbf{r}_t(Q_t)\!=\!(R_t\odot Q_t)\mathbf{1}.
\end{equation}


\begin{theorem} 
The max-min problem given by (\ref{eq:Qt}) with (\ref{eq:MR}) can be solved by  the following equivalent linear programming problem (given $t$, $\mathbf{d}$, $G_{t,k} \text{ for } k=1\dots p$, $R_t$, and $\hat{U}_{t+1}$):
\begin{equation}\label{eq:LP}
\begin{aligned}
\underset{Q,\mathbf{y},z,\mathbf{r}, M,S,\mathbf{s}, K}{\text{maximize}}& \hspace*{0.5cm} -\mathbf{d}^T\mathbf{y}+z\\
\text{subject to}&\hspace*{0.5cm} M=\sum _{k=1}^{p}G_{t,k}\odot \left(\mathbf{1}(Q\mathbf{e}_k)^T\right)\\
&\hspace*{0.5cm} \mathbf{r}=(R_t\odot Q)\mathbf{1}\\
&\hspace*{0.5cm} -\mathbf{y}+z\mathbf{1}\leq \mathbf{r}+M^T\hat{U}_{t+1}\\
&\hspace*{0.5cm} K=M+S+\mathbf{s}\mathbf{1}^T\\ 
&\hspace*{0.5cm} \mathbf{s}+\mathbf{d}\geq K\mathbf{d}\\
&\hspace*{0.5cm} Q\mathbf{1}=\mathbf{1}, Q\geq 0, \mathbf{y}\geq 0, S\geq 0, K\geq 0.
\end{aligned}
\end{equation}
\end{theorem}
\begin{proof}
The proof will use the duality theory of linear programming \cite{Boyd:2004:CO:993483}, which implies that the following primal and dual problems produce the same cost  
\begin{align*}
\text{PRIMAL}  \hspace*{1.4cm} & \hspace*{2cm} \text{DUAL} \\
 \text{minimize }\mathbf{b}^T\mathbf{x} \hspace*{1cm} & 
  \hspace*{1.5cm}  \text{maximize }\mathbf{c}^T\mathbf{y}\\
\text{ s.t. } A^T\mathbf{x}\geq \mathbf{c}, \mathbf{x}\geq 0 \hspace*{0.5cm}  &  \hspace*{1.2cm}  \text{ s.t. } A\mathbf{y}\leq \mathbf{b}, \mathbf{y}\geq 0.
\end{align*} 
 Since the set $\mathcal{X}$ is defined as $0\leq \mathbf{x}\leq \mathbf{d}$ and $\mathbf{x}^T\mathbf{1}=1$, then the $\min$ in \eqref{eq:Qt} can be obtained by a minimization problem with the following primal problem parameters:
$$\mathbf{b}=U_t, A= [-I_n \ \mathbf{1} \ -\mathbf{1}], \text{ and }\mathbf{c}^T=[-\mathbf{d}^T \ 1 \ -1].$$
The dual of this program  is 
\begin{align*}
\underset{\mathbf{y},z}{\text{maximize}}& \hspace*{0.5cm} -\mathbf{d}^T\mathbf{y}+z\\
\text{subject to}&\hspace*{0.5cm} -\mathbf{y}+z\mathbf{1}\leq U_t(Q_t)\\
&\hspace*{0.5cm} \mathbf{y}\geq 0, z \text{ unconstrained}.
\end{align*}
Next by considering the $\text{argmax}$ in \eqref{eq:Qt}, it remains to show that the set $\mathcal{C}$ can be represented by linear inequalities to write \eqref{eq:Qt} as a maximization LP problem. It is indeed the case by  using \cite[Theorem 1]{behcet_tac15} which says the following:
\begin{align*}
M\in \mathcal{M} \hspace*{0.5cm} &\Leftrightarrow  \hspace*{0.5cm} \exists S\geq 0, K\geq 0, \mathbf{s} \text{ such that }\\
~ & \hspace*{1cm}  K=M+S+\mathbf{s}\mathbf{1}^T\\ 
& \hspace*{1cm} \mathbf{s}+\mathbf{d}\geq K\mathbf{d},
\end{align*}
where $\mathcal{M}=\underset{\mathbf{x}\in \mathcal{X}}{\cap}\mathcal{M}(\mathbf{x})$ and $\mathcal{M}(\mathbf{x})\!=\!\{M\! \in \mathbb{R}^{n\times n}, \mathbf{1}^TM\!=\! \mathbf{1}^T, M\geq 0, M\mathbf{x}\leq \mathbf{d}\}$.
 As $M$ in \eqref{eq:MR} is a linear function of the decision variable $Q$, the set $\mathcal{C}$ is equivalently described by  $\mathcal{C}=\{Q\in \mathbb{R}^{n,p},Q\mathbf{1}=\mathbf{1},Q\geq 0, M(Q)\in \mathcal{M}\}$, which implies that $\mathcal{C}$ can be described by linear inequalities. Now combining this result  with the dual program, we can conclude that  $\hat{Q}_t$ can be obtained via the linear program given in the theorem, which concludes the proof. 
\end{proof}
\subsection{Better Heuristic for $\hat{Q}_t$}
The resulting solution $\hat{Q}_t$ of linear program \eqref{eq:LP} is not unique. Let the convex solution set be  $\mathcal{Q}$. Therefore, among the possible solution variables $Q\in \mathcal{Q}$ we are interested in values that are as close as possible to $Q_{MDP}$ (found by Algorithm \ref{alg1}) because the latter gives optimal solution policy for unconstrained MDP problem. But since $Q_{MDP}$ might not be feasible (due to the additional constraints), we target the projection of $Q_{MDP}$ on $\mathcal{Q}$. Therefore, we choose $\hat{Q}_t$ to be the solution of the following optimization:
\begin{align*}
\hat{Q}_t=\underset{Q\in \mathcal{Q}}{\text{argmin}}& \hspace*{0.5cm} ||Q-Q_{MDP}||.
\end{align*}
Note that if $Q_{MDP}\in \mathcal{Q}$, then this optimization will give the optimal policy. Therefore, with this extra optimization, the output policy not only guarantees a lower bound on expected reward reward, but it also retrieves back the solution of the unconstrained MDP if the state constraints were relaxed. 

\section{Simulations}
This section presents  a simulation example to demonstrate the performance of the proposed methodology for CMPDs on a vehicle swarm coordination problem \cite{Acikmese:2012Ama, behcet_tac15}. 
 In this application, autonomous vehicles (agents)  explore a region $F$, which can be partitioned into $n$ disjoint subregions (or bins) $F_i$ for  $i=1, \dots, n$ such that $F=\cup _iF_i$.  We can model the system as an MDP where the states of agents are their  bin locations and the actions of a vehicle are defined by the possible transitions to neighboring bins.  Each vehicle collects rewards while traversing the area where, due to the stochastic environment, transitions are stochastic (i.e., even if the vehicle's command is to move to ``right", the environment can send the vehicle to ``left"). 
Note that  the state constraints discussed in this paper can can be interpreted as follows. If a large number of vehicles are used, then  the density of vehicles evolve as a Markov chain. Since the physical environment (capacity/size of bins) can impose  constraints on the number of vehicles in a given bin, the state (safety) constraints on the density are needed.

For simplicity we consider  the operational region to be a $3$ by $3$ grid. Each vehicle has $5$ possible actions,  ``up'',  ``down'', ``left'',  ``right'', and ``stay'', see Figure \ref{bins}. When the vehicle is on the boundary, we set the probability of actions that cause transition outside of   the domain to zero. For example  in bin 1 the actions ``left" and ``up" are not permitted, which can easily be imposed as linear equality constraints  in our formulation. 

%
The reward vectors $R_t$ for $t=1,\dots ,N\!-\!1$ and $R_N$ are (tenth state is not icn, 
\begin{equation}
\begin{split}
R_t = \left[ \begin{array}{ccccccccc}1 & 1 & 1 & 10 & 5 & 0 & 3 &  3 & 3    \end{array}\right]^T
\text{ and } \\
R_N = \left[ \begin{array}{ccccccccc}0 & 0  & 0  & 10 & 0 & 0 &  0 & 0 & 0    \end{array}\right]^T 
\end{split}
\end{equation}
where $R[i]$ is the reward collected at bin (state) $i$  and is assumed independent of the action taken.
Density constraints for different bins are given as follows
\begin{equation}
 \mathbf{d}=\left[ \begin{array}{ccccccccc} 0.4 & 0.4  & 0.4 & 0.5 & 0.05 & 1 &  0.2 & 0.2& 0.2    \end{array}\right]^T,
\end{equation}
where any bin $i$ should have $x_t[i]\leq d[i]$ for $t=1, 2, \dots$.
The MDP solution (which is known to give deterministic policies) that maximizes the total expected reward does not satisfy these constraints. However, with our proposed policy, not only the constrained are satisfied, but also the solution gives guarantees on the expected total reward. Note that the linear program generates the policies independent of the initial distribution. Therefore, even if the latter was unknown (which is usually the case in autonomous swarms), the generated policy satisfy the constraints.
\begin{figure}
\begin{center}
\includegraphics[scale=0.14]{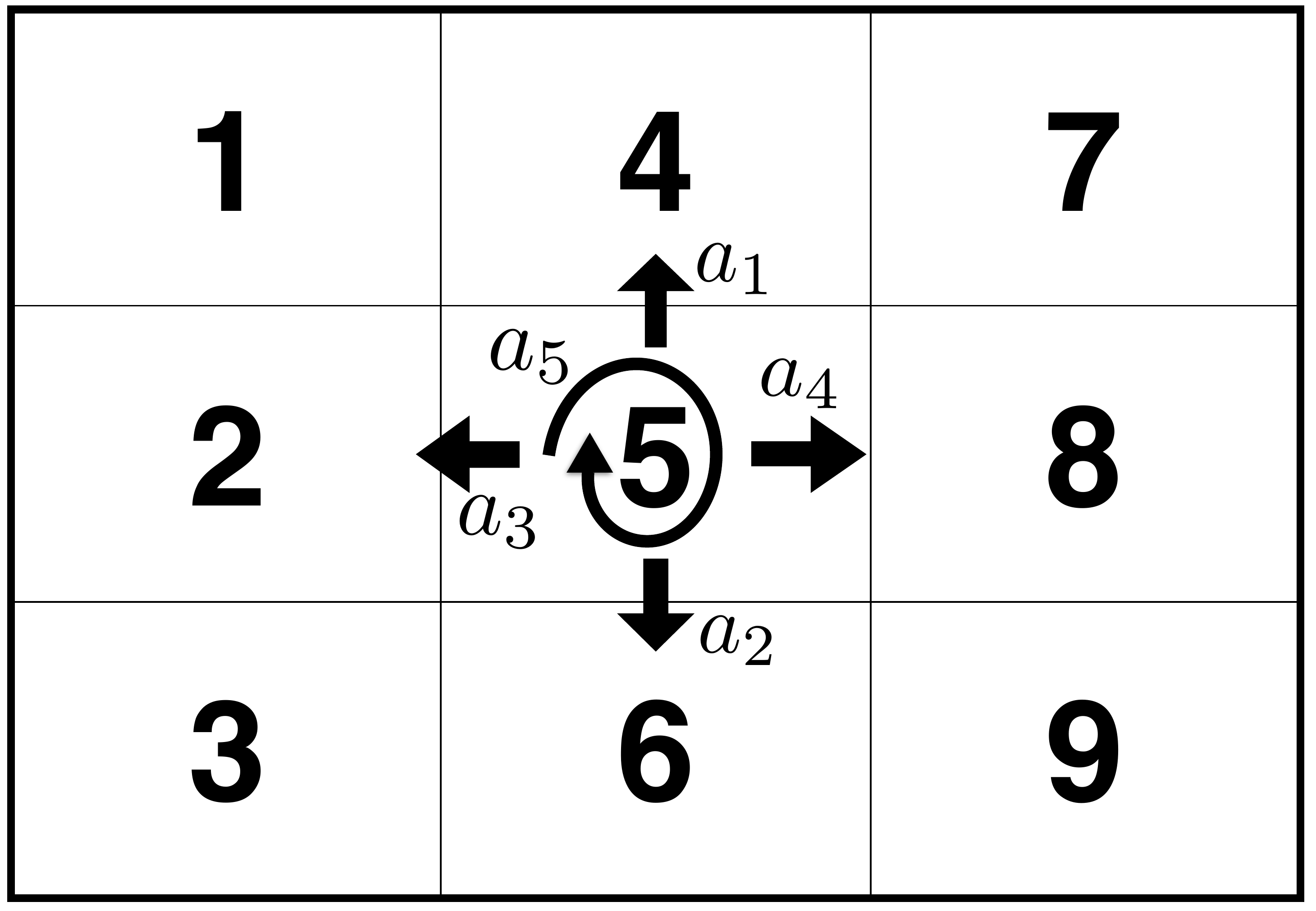}
\caption{Illustration of $3 \times 3$ grid describing the MDP states, and the  5 actions  (Up, Down, Left, Right, and Stay).
} 
\label{bins}
\vspace{-20pt}
\end{center}
\end{figure}

We now consider that all the vehicles initially are in bin $6$ (i.e., $x_{t=1}[6]=1$ and note that this is a feasible starting vector because $d[6]=1$).  Figure \ref{constraints} shows that in the scenario considered in this simulation, the unconstrained MDP policy would lead all the swarm vehicle to one bin and it would violate the constraint because the maximum allowed density is $0.5$. The constraints are also violated at the bin $5$ as optimal policy made the swarm traverse this bin leading to a density of $0.8$ where the maximum allowed density was $0.05$. However, our policy generated from Algorithm \ref{alg3} led to a distribution of the swarm in such a way the constrains are satisfied at every iteration. To further investigate the efficiency of the algorithm we have to study the rewards associated to the proposed policy.

\begin{figure}
\begin{center}
\includegraphics[scale=0.41]{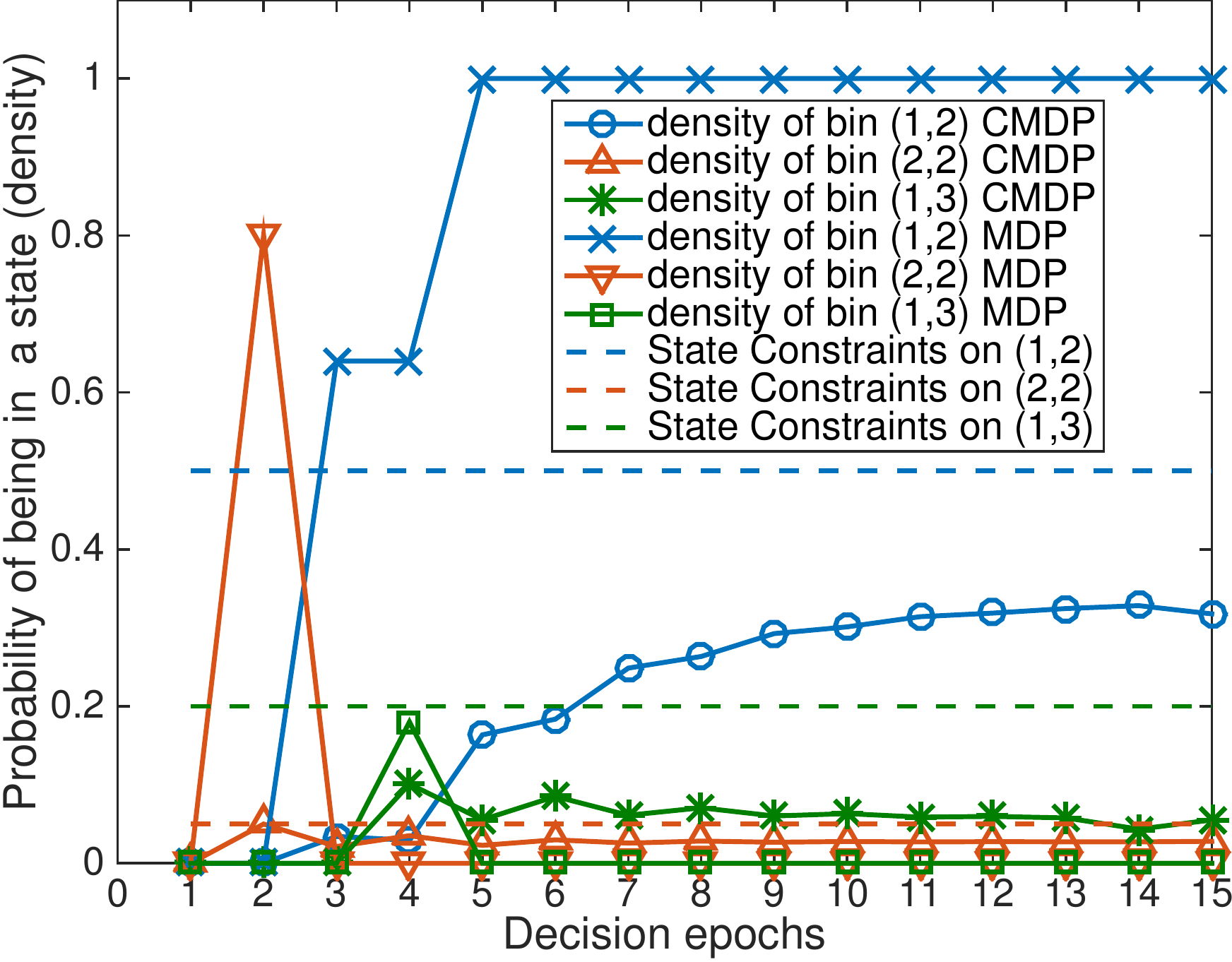}
\caption{The figure shows the density of autonomous vehicles and how the policy for unconstrained MDP can violate the constraints. The density for bin $4$ jumped all the way to 1  after 5 iterations while its maximum capacity is 0.5. The synthesized CMDP policy    obeys the constraints while giving a lower bound guarantee on the expected  reward.
} 
\label{constraints}
\end{center}
\end{figure}

In figure \ref{rewards} we show the reward of the constrained MDP policy  and we compare it to the unfeasible policy of unconstrained MDP. It turns out that in this scenario, the added heuristic generated by the proposed methods in this paper could achieve closer reward to the maximum possible reward without constraints. The constrained MDP curve in crossed line (yellow) is the lower bound derived by Theorem $\ref{thm:1}$ which providing optimality guarantees for the LP generated policy.

\begin{figure}
\begin{center}
\includegraphics[scale=0.4]{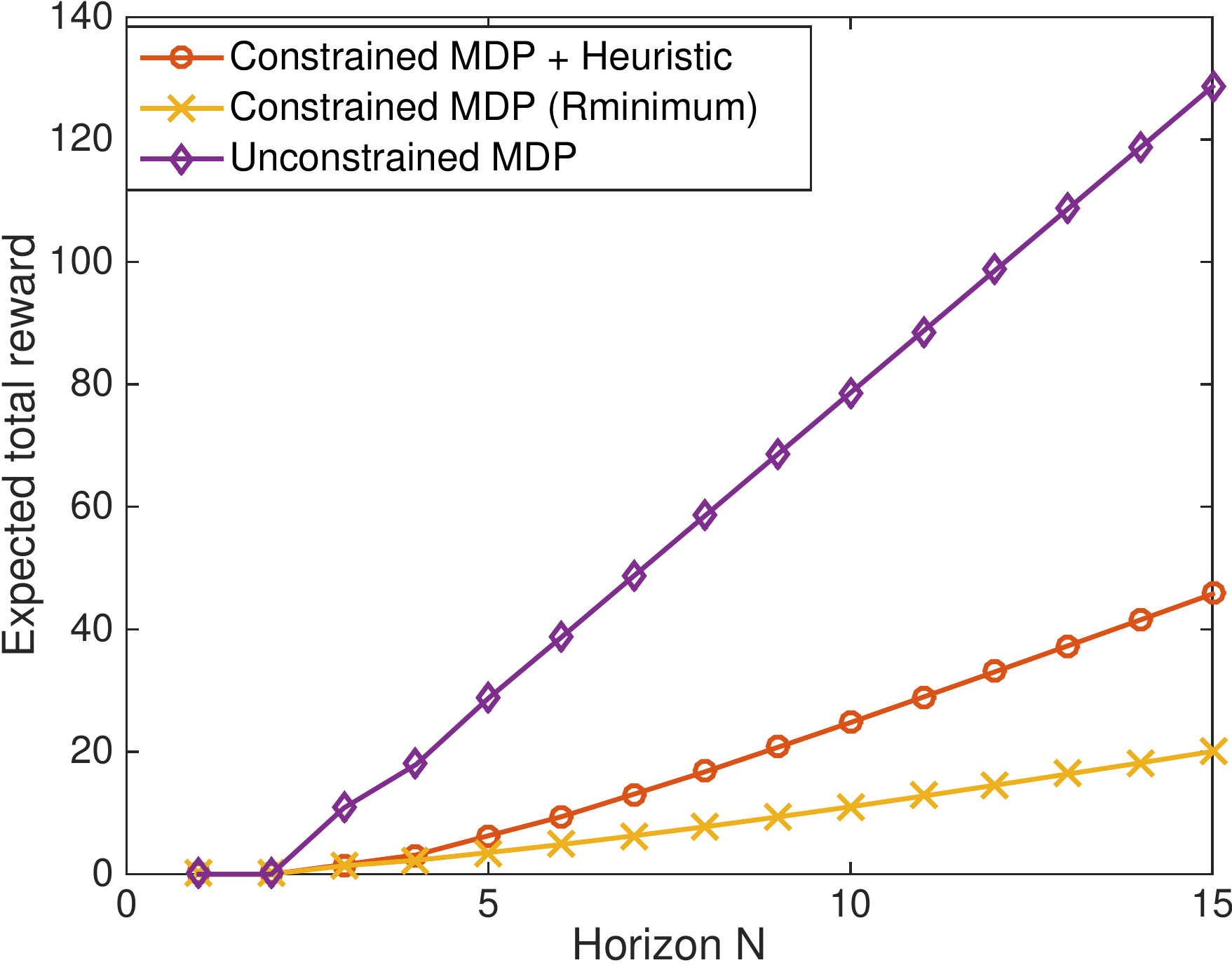}
\caption{ The curve corresponding to the unconstrained MDP is the total expected reward for the optimal MDP  policy without considering the state constraints. Of course the policy is unfeasible and cannot be used when constraints are present. The constrained MDP is the reward  corresponding to the policy computed by the linear program \eqref{eq:LP} (it is the computed lower bound on the total expected reward). The constrained MDP plus heuristic is the further enhancement obtained  by projecting the optimal deterministic MDP on the set of feasible policies for CMDP with reward guarantees. 
} 
\label{rewards}
\vspace{-10pt}
\end{center}
\end{figure}
\section{Conclusion}
In this paper, we have studied finite-state finite-horizon MDP problems with state  constraints. It is  shown that policies due to unconstrained MDP algorithms are not feasible and we propose an efficient algorithm based on linear programming and duality theory to generate feasible Markovian policies that not only satisfy the constraints, but also provide some guarantees on the expected reward. This new policy defines a probability distribution over possible actions and requires that agents randomize their actions depending on the state. 
In the absence of constraints, the proposed method retrieves back the optimal standard MDP policies.
For future work, we would like to extend the proposed policy for the infinite-horizon case using a similar algorithm as the ``value iteration'' of standard MDP problems.

\bibliographystyle{./IEEEtran}
\bibliography{./CMDP_bib,./IEEEabrv}
\end{document}